\documentclass[letterpaper, 10pt, conference]{ieeeconf}
\IEEEoverridecommandlockouts
\usepackage{cite}
\usepackage{amsmath,amssymb,amsfonts}
\usepackage{algorithmic}
\usepackage{graphicx}
\usepackage{textcomp}
\usepackage{amssymb}
\usepackage{amsmath}
\usepackage{mathtools}
\usepackage{xcolor}
\usepackage{cite}
\usepackage{accents}
\newcommand{\ubar}[1]{\underaccent{\bar}{#1}}
\usepackage[ruled,vlined,linesnumbered]{algorithm2e}

\newcommand{\lev}{\textrm{lev}}
\newcommand{\olsi}[1]{\,\overline{\!{#1}}} 

\usepackage{longtable}
\usepackage{array}
\usepackage{tabularx}
\usepackage{booktabs}
\usepackage{tasks}
\usepackage{authblk}
\newtheorem{theorem}{Theorem}

\newtheorem{assumption}{Assumption}

\newtheorem{lemma}{Lemma}

\title{\LARGE \bf
Reduced Sample Complexity in Scenario-Based Control System Design via Constraint Scaling
}
\author{Jaeseok Choi$^{1}$, Anand Deo$^{2}$, Constantino Lagoa$^{3}$ and Anirudh Subramanyam$^{1}$
\thanks{This work was supported in part by the U.S. Department of Energy under Grant DE-SC0023361.}
\thanks{$^{1}$J. Choi and A. Subramanyam are with the Department of Industrial and Manufacturing Engineering, The Pennsylvania State University, University Park, PA, USA (e-mail: subramanyam@psu.edu)}
\thanks{$^{2}$A. Deo is with the Decision Sciences Area, Indian Institute of Management Bangalore, Bengaluru, India.}
\thanks{$^{3}$C. Lagoa is with the School of Electrical Engineering and Computer Science, The Pennsylvania State University, University Park, PA, USA.}
}

\begin{document}

\maketitle
\thispagestyle{empty}
\pagestyle{empty}

\begin{abstract}
    The scenario approach is widely used in robust control system design and chance-constrained optimization, maintaining convexity without requiring assumptions about the probability distribution of uncertain parameters. However, the approach can demand large sample sizes, making it intractable for safety-critical applications that require very low levels of constraint violation. To address this challenge, we propose a novel yet simple constraint scaling method, inspired by large deviations theory. Under mild nonparametric conditions on the underlying probability distribution, we show that our method yields an exponential reduction in sample size requirements for bilinear constraints with low violation levels compared to the classical approach, thereby significantly improving computational tractability. Numerical experiments on robust pole assignment problems support our theoretical findings.
\end{abstract}


\section{Introduction} \label{sec:introduction}
Chance-constrained optimization has emerged as an effective framework for designing controllers in the presence of uncertainty. Unlike traditional robust control design, which aims to satisfy the system constraints for all admissible realizations of the uncertain parameters, chance constraints allow for a small yet controlled probability of constraint violation, thus circumventing the conservatism of robust approaches. Unfortunately, however, the solution of chance-constrained optimization problems is computationally intractable even in simple cases~\cite{luedtke2010integer}.  Several algorithms have been proposed in the literature for their solution, including those based on sample-free analytical methods~\cite{prekopa2003,chen2010cvar,lagoa2005probabilistically,tong2022optimization,bardakci2023probability}, sample average approximations~\cite{hong2011sequential,nemirovski2007convex,pena2020solving}, importance sampling~\cite{vanAckooij2014,barrera2016chance}, linear matrix inequalities~\cite{jasour2015semidefinite},
and scenario approaches~\cite{calafiore2006scenario,nemirovski2006scenario,blanchet2024efficient}.

Among the various algorithms, the scenario approach remains one of the most widely used.
The key idea is to enforce the system constraints for only a finite number of randomly generated samples of the uncertain parameters and solve the corresponding ``sampled optimization problem'' to obtain the desired design parameters.
The popularity of this approach can be attributed to its simplicity and computational efficiency.
In particular, if the original constraints are convex in the design variables, then the sampled problem is also convex.
Moreover, bounds on the number of required samples can be easily computed as a function of the problem dimension, confidence level of the randomization procedure, and the prescribed probability of constraint violation, $\varepsilon$, across all (potentially non-sampled) realizations of the uncertainty.
Examples of such bounds can be found in~\cite{calafiore2005uncertain, calafiore2006scenario, campi2008exact, campi2009scenario}.

Our work is motivated by applications that necessitate very high levels of reliability ($>99\%$), or equivalently, very low probabilities of constraint violation $\varepsilon$.
Examples can be found in aircraft control~\cite{HORN20051037}, robotics~\cite{janson2018monte}, autonomous driving~\cite{de2021risk}, power systems~\cite{lukashevich2024importance} and telecommunications~\cite{alsenwi2019chance}.
In these applications, the required sample complexity of the scenario approach, which scales roughly as $\varepsilon^{-1}$, can become extremely large when the desired probability of constraint violation is small.
For example, when $\varepsilon = 0.001$, which translates to a modest reliability level of $99.9\%$,
the bound from~\cite{campi2009scenario} requires the introduction of 7,992 samples to guarantee feasibility of the obtained solution with 95\% confidence, even when there is only one design parameter.
This large sample requirement severely reduces the computational efficiency of the scenario approach in applications that either need high levels of reliability, or are limited by computational resources, or need to be deployed in online settings such as model predictive control.

Our work aims to drastically reduce the sample complexity of the scenario approach.
In this letter, we restrict our attention to design problems specified by multiple (joint) constraints, each of which is a bilinear function of the design variables and uncertain parameters.
Such bilinear constraints arise in a wide variety of control design problems, including linear model predictive control with additive disturbances~\cite{cannon2010stochastic,kouvaritakis2010explicit}, robust superstable control~\cite{blanchini2000convex}, and robust pole assignment~\cite{keel1997linear}.
Our main contribution is a ``scaled sampled problem'' in which the constant right-hand side terms of the bilinear constraints are scaled down by a prespecified factor $s \geq 1$.
Under a mild nonparameteric assumption on the probability distribution of the uncertainties (satisfied by a large class of distributions), we show that the sample complexity of the scaled sampled problem is exponentially smaller.
Specifically, we demonstrate
that as the required constraint violation level is made more stringent ($\varepsilon \to 0$), the sample complexity of the scaled problem increases only as $\varepsilon^{-s^{-\alpha}}$, where $\alpha > 0$ is a parameter that depends only on the underlying probability distribution.
In the previous example where the desired violation level is $\varepsilon = 0.001$ and the underlying distribution is multivariate normal ($\alpha = 2$), a modest scaling of $s = 1.2$ requires the introduction of only $969$ samples, consituting an eight-fold reduction in sample size requirement.

The idea of scaling in scenario approaches was originally conceived, to the best of our knowledge, in~\cite{nemirovski2006scenario}, where the authors propose to scale the samples before solving the sampled problem.
Their approach requires finding and tuning an alternate so-called majorizing distribution that possess certain specific properties, which does not appear to readily generalize to a wide class of problems.
Recently, the authors of~\cite{blanchet2024efficient} propose to augment the scenario approach for problems with heavy-tailed distributions.
Their method requires certain analytically computable outer approximations to the uncertain system constraints, which renders it highly application-specific. Finally, under assumptions similar to ours, 
\cite{blanchet2024optimization} recently derived asymptotic relationships between the chance-constrained and sampled problems, without addressing sample complexity or algorithmic issues.
In addition to these probabilistic approaches, other techniques to reduce the sample complexity of scenario approaches include both application-specific\cite{kohler2020computationally,schildbach2014scenario,lukashevich2024priori} and general constraint removal methods \cite{campi2011sampling, romao2022exact}; the latter enable lower costs but at the expense of solving multiple optimization problems.
In contrast to the aforementioned works, our proposed method aims to provide a more general recipe for reducing sample complexity in scenario approaches. 
Our method is practical, easy to implement, and almost identical to the original scenario approach. Moreover, it requires neither extensive parameter tuning nor computation of analytical approximations, and performs well under very general assumptions about the underlying probability distribution, making it applicable to a wide range of uncertainty structures.

The remainder of the letter proceeds as follows. In Section~\ref{sec:formulation}, we present the problem formulation we study along with assumptions and briefly review the background on the scenario approach. In Section~\ref{sec:results}, we present our proposed method and the key theoretical results underpinning the method.
Finally, validation of these results via numerical experiments are provided in Section~\ref{sec:experiment}.

\emph{Notation.} We use $\mathbb{R}_{+}$ to denote the set of nonnegative reals. For a function $f$, denote the upper level set $\{z: f(z) > a,\, \|z\|\leq M\} =\lev_{a,M}(f)$. If $M=\infty$, we instead write $\lev_a(f)$. Further, let $\bar A$ denote the closure of a set $A$. We use Landau's notation for orders: for functions $f$ and $g$, we say that $f(t) = O(g(t))$ if $f(t) \leq c_0g(t)$ for some $c_0$, and say that $f(t) =o(g(t)) $ if $f(t)/g(t) \to 0$ as $t\to 0$.

\section{Problem Description and Background} \label{sec:formulation}
We consider the chance-constrained optimization problem,
\begin{equation}
    \label{eq:ccp}
    \tag{$\text{CCP}_\varepsilon$}
    \begin{aligned}
        \min_{x \in X} & \quad c(x)             \\
        \text{s.t.}    & \quad \mathbb{P}_{\xi}
        \left(
        \max_{i=1,\ldots,m} x^\top A^{(i)}\xi \leq 1
        \right) \geq 1 - \varepsilon,
    \end{aligned}
\end{equation}
where $x\in\mathbb{R}^n$ is the vector of design parameters, $X \subseteq \mathbb{R}^n$
is a compact and convex set, $c:\mathbb{R}^{n}\mapsto \mathbb{R}$ is a convex function, and $\xi\in\mathbb{R}^{d}$ is the vector of uncertain parameters. 
There are $m$ constraints that must be jointly satisfied with probability at least $1 - \varepsilon$, where $\varepsilon \in (0,1)$ denotes the prescribed violation level.
Each constraint is linear in $x$ for fixed $\xi$ and linear in $\xi$ for fixed $x$, and the corresponding matrix of coefficients for the $i$th constraint is denoted by $A^{(i)} \in \mathbb{R}^{n \times d}$.
We shall assume that problem~\eqref{eq:ccp} has at least one feasible solution for all $\varepsilon \in (0,1)$, ensuring that its solutions are well-defined.

Central to our development is the following assumption, where we require that $\xi$ is a continuous random vector with a density function satisfying the nonparametric condition of multivariate regular variation~\cite{basrak2002characterization}.
Before stating the assumption, we note that a function $h:\mathbb R_{+} \mapsto \mathbb R_{+}$ is said to be `slowly varying' if $\lim_{t\to\infty}h(tx)/h(t) = 1$ for all $x>0$.
\begin{assumption}\label{assumption:rv}
    The density function of $\xi$ satisfies $f_\xi(z) = \exp(-Q(z))$, where $Q: \mathbb{R}^d \mapsto \mathbb{R}$ is a multivariate regularly varying function.
    Specifically, there exists a constant $\alpha>0$, a slowly varying function $h$, and a continuous function $\lambda : \mathbb{R}^d_{+} \mapsto \mathbb{R}_{+}$, such that $q(u) \coloneqq h(u) u^\alpha$ is increasing and continuous for sufficiently large $u$, and
    \begin{equation}\label{eq:lambda}
        \lim_{u\to\infty}\frac{Q(uz)}{q(u)} = \lambda(z) \text{ for all } z \in \mathbb{R}^d_{+}.
    \end{equation}
\end{assumption}
The family of probability distributions that satisfy Assumption~\ref{assumption:rv} can be interpreted as those whose density function is roughly equal to the exponential of a polynomial.
This family encompasses a wide range of distributions, including light-tailed distributions, such as multivariate normal ($\alpha=2$) and exponential ($\alpha=1$), heavy-tailed distributions ($\alpha < 1$), and their mixtures. The parameter $\alpha$ can also be estimated from data~\cite{de2016approximation}.
Moreover, unlike existing literature on chance-constrained optimization~\cite{barrera2016chance,nemirovski2006scenario}, the assumption does not require the existence of a finite moment generating function, which may not exist even for simple elliptical distributions.
Table~\ref{table:distributions} illustrates the parameter $\alpha$ and limiting function $\lambda(\cdot)$ for some examples of distributions satisfying Assumption~\ref{assumption:rv}. See \cite{deo2023achieving} for further details.

\begin{table}[!ht]
    \centering
    \caption{Examples of distributions satisfying Assumption~\ref{assumption:rv}}
    \label{table:distributions}
    \begin{tabular}{p{3.5cm}|p{0.5cm}|p{1.5cm}}
        \toprule
        Distribution family                                                    & $\alpha$ & $\lambda(z)$                                                         \\
        \midrule\midrule
        Multivariate normal$^1$                                                & $2$      & $ \displaystyle \frac12 \left(z^\top \Sigma^{-1} z \right)$          \\ \midrule
        Elliptical$^1$ with generator $R$ such that $f_R(r) = \exp(-r^{k})$,   & $k$      & $\displaystyle (z^\top \Sigma^{-1}z)^{k}$                            \\ \midrule
        Gaussian mixture$^2$                                                   & $2$      & $\displaystyle \frac{1}{2}\sum_{k=1}^K (z^\top \Sigma^{-1}_k z)^{2}$ \\ \midrule
        Weibull with shape $k$ and scale parameters $\sigma_1,\ldots,\sigma_m$ & $k$      & $\displaystyle \sum_{i=1}^m \left(\frac{z_i}{\sigma_i}\right)^k$     \\
        \bottomrule
        \multicolumn{3}{l}{\footnotesize $^1$Covariance $\Sigma$,\quad $^2 K$ components with covariances $\Sigma_1, \ldots, \Sigma_K$}
    \end{tabular}
\end{table}

\subsection{Overview of the Scenario Approach}
For any $x \in X$, let $V(x) \coloneqq \mathbb{P}_{\xi} ( \max_{i=1,\ldots,m} x^\top A^{(i)}\xi > 1 )$ denote the probability of constraint violation of $x$.
Therefore, any feasible solution $x$ of~\eqref{eq:ccp} satisfies $V(x) \leq \varepsilon$ and is also termed an ``$\varepsilon$-level solution''.
To solve~\eqref{eq:ccp}, the standard scenario approach~\cite{campi2009scenario} formulates a sampled problem, obtained by replacing the probabilistic constraint with a finite number of deterministic constraints, each corresponding to a sample, $\xi^{(j)}$, $j = 1, \ldots, N$, of the uncertain parameter vector $\xi$, randomly generated according to $\mathbb{P}_{\xi}$.
\begin{equation}
    \label{eq:sp}
    \tag{$\text{SP}_{N}$}
    \begin{aligned}
        \min_{x\in X} & \quad c(x)                                                                     \\
        \text{s.t.}   & \quad \max_{i=1,\ldots,m} x^\top A^{(i)} \xi^{(j)} \leq 1, \quad j=1,\ldots,N.
    \end{aligned}
\end{equation}
While the scenario approach extends to general convex constraints, we focus on problems with bilinear constraints that also commonly occur in various control applications. In particular,
it is known (see \cite[Theorem 1]{campi2009scenario}), that if
\begin{align}\label{eq:numsample}
    N \geq \ubar{N}(\varepsilon,\beta) \coloneqq  \left\lceil\frac{2}{\varepsilon} \left( \log \frac{1}{\beta} + n \right)\right\rceil,
\end{align}
then with probability at least $1-\beta$ (with respect to the $N$-fold product distribution $\mathbb{P}_{\xi}^N$),
the sampled problem~\eqref{eq:sp} is either infeasible or any feasible solution $\hat{x}_N$ of~\eqref{eq:sp} is an $\varepsilon$-level feasible solution of~\eqref{eq:ccp}.

\section{Proposed Method: Scaled Sampled Problem}\label{sec:results}
As we noted in the introduction, the sample size prescribed by equation~\eqref{eq:numsample} can become extremely large when the desired level of violation, $\varepsilon$, is small.
To circumvent this sample complexity, we introduce the following ``scaled sampled problem'' obtained by scaling down the right-hand side coefficients of the constraints in the sampled problem~\eqref{eq:sp} by a factor $s \geq 1$.
\begin{equation}
    \label{eq:s-sp}
    \tag{$\text{SSP}_{N, s}$}
    \begin{aligned}
        \min_{x\in X} & \quad c(x)                                                                               \\
        \text{s.t.}   & \quad \max_{i=1,\ldots,m} x^\top A^{(i)} \xi^{(j)} \leq \frac{1}{s}, \quad j=1,\ldots,N.
    \end{aligned}
\end{equation}
Observe that the scaled sampled problem~\eqref{eq:s-sp} reduces to the original sampled problem~\eqref{eq:sp} when $s = 1$.
Intuitively, the scaled version~\eqref{eq:s-sp} restricts the feasible region of~\eqref{eq:sp} and therefore, feasible solutions of the former must also be feasible in the latter whenever the number $N$ of sampled constraints and the corresponding realizations, $\xi^{(1)}, \ldots, \xi^{(N)}$, are identical.
However, our key insight is that one only needs to solve the scaled sampled problem~\eqref{eq:s-sp} for a much smaller number of samples, $N$, compared to the prescription, $N(\varepsilon, \beta)$, provided by equation~\eqref{eq:numsample}.

Our proposed procedure is summarized in Algorithm~\ref{alg}.
\begin{algorithm}[!ht]
    \caption{Scaled scenario approach}\label{alg}
    \KwIn{Constraint violation level $\varepsilon$, confidence parameter $\beta$, distribution parameter $\alpha$, scaling factor $s \geq 1$}
    \KwOut{Approximate solution, $\hat{x}_{N,s}$, of~\eqref{eq:ccp}}
    \SetAlgoLined
    Sample $N = \ubar{N}(\varepsilon^{s^{-\alpha}},\beta)$ independent realizations, $\xi^{(1)},\ldots,\xi^{(N)} \sim \mathbb{P}_{\xi}$\\
    Solve~\eqref{eq:s-sp} and set $\hat{x}_{N,s}$ as its solution
\end{algorithm}
When compared with the standard scenario approach, our procedure introduces only one additional user-defined parameter, namely the scaling factor $s$.
Larger values of $s$ directly translate to fewer required samples in the scaled sampled problem and to reduced computation times for its solution.
In particular, for small values of $\varepsilon$, our procedure requires roughly $s^{\alpha}$ times fewer samples compared to the standard scenario approach, as can also be immediately verified from the following equation.
\begin{align}\label{expo_reduction}
    \lim_{\varepsilon \to 0^{+}} \frac{\log \ubar{N}(\varepsilon^{s^{-\alpha}}, \beta)}{\log \ubar{N}(\varepsilon, \beta)} = \frac{1}{s^\alpha}
\end{align}
Intuitively, this is because our procedure only requires as many samples as are required by the standard scenario approach to obtain an $\varepsilon_1$-level feasible solution of the original problem~\eqref{eq:ccp}, where $\varepsilon_1 = \varepsilon^{s^{-\alpha}}$ satisfies $\varepsilon_1 \geq \varepsilon$.
We highlight, however, that the lower sample complexity may come at a potential increase in objective value.
The scaling factor $s$ can thus be interpreted as offering a smooth tradeoff between computational efficiency and solution conservatism.

The following theorem provides theoretical justification for the efficiency of our method (see Appendix for proof).
\begin{theorem}[Asymptotic feasibility of \eqref{eq:s-sp}]\label{thm:weak_feasibility}
    Fix $s \geq 1$ and $\beta \in (0,1)$.
    Let $\{x_{\varepsilon}\}_{\varepsilon>0}$ denote any sequence of feasible solutions of the scaled sampled problem~\eqref{eq:s-sp} with $N \geq \ubar{N}(\varepsilon^{s^{-\alpha}},\beta)$.
    Then, with probability at least $1-\beta$,
    \begin{equation}\label{eq:asymptotic_feasibility}
        \liminf_{\varepsilon \rightarrow 0^{+}}\frac{\log V(x_{\varepsilon})}{\log \varepsilon} \geq 1.
    \end{equation}
\end{theorem}
We note that although the scaled sampled problem \eqref{eq:s-sp} may become infeasible for some choice of $N$ samples and scaling factor $s$, Theorem~\ref{thm:weak_feasibility} implicitly considers only those cases for which the infinite sequence $\{x_{\varepsilon}\}_{\varepsilon > 0}$ is well-defined. In such cases,
equation~\eqref{eq:asymptotic_feasibility} implies that, for any arbitrarily chosen $\delta > 0$, we have
\[
    \frac{\log V(x_{\varepsilon}) }{ \log \varepsilon } \geq 1-\delta,
\]
for all sufficiently small $\varepsilon$.
In other words, when $\varepsilon$ is small, the violation probability of solutions obtained from our procedure is almost equal to the desired level, $V(\hat{x}_{N,s}) = \varepsilon^{1+o(1)}$.
Theorem~\ref{thm:weak_feasibility} combined with equation~\eqref{expo_reduction} thus shows that
our scaled approach can obtain approximate $\varepsilon$-level feasible solutions of the original \eqref{eq:ccp} similar to the standard scenario approach, but with exponentially fewer samples, $\ubar{N}(\varepsilon^{s^{-\alpha}},\beta)$ instead of $\ubar{N}(\varepsilon,\beta)$.
Our experiments indicate that, in practice, the violation probability remains far below the desired target even for modest levels of $\varepsilon$.

\section{Numerical Validation} \label{sec:experiment}

We demonstrate the performance of our proposed constraint scaling scenario approach on the problem of robust pole assignment problem. The presented example is a small modification of the one in~\cite{lagoa2005probabilistically}. Consider a continuous uncertain open loop plant described by the transfer function\footnote{Although $s$ is conventionally used to denote the Laplace transform variable for continuous-time transfer functions, we use the symbol $z$ here to avoid confusion with our constraint scaling parameter $s$.},
\[
    G_{\xi}(z)=\frac{n_{\xi}(z)}{d_{\xi}(z)}=\frac{(0.75+\xi_3)z+1.25+\xi_4}{z^2+(0.75+\xi_1)z+\xi_2},
\]
where $G_{0}(z) = (0.75z+1.25)/(z^2+0.75z)$ represents the nominal transfer function and the vector $\xi=(\xi_1,\ldots,\xi_4)$ represents the uncertainty in the coefficients of the numerator and denominator.
The presence of uncertainty
makes it challenging to achieve exact closed-loop specifications.
We instead attempt to design a controller of the form\footnote{This is a lead compensator when we add the constraint, $x_2 < x_1$, and a lag compensator if we instead add $x_1 < x_2$.},
\[
    G_c(z)=\frac{n_c(z)}{d_c(z)}=\frac{x_1z+x_2}{z+1},
\]that encourages minimizing the control effort, $x_1^2 + x_2^2$, while ensuring that
each coefficient of the closed loop characteristic polynomial,
\(
n_{\xi}(z)n_c(z)+d_{\xi}(z)d_c(z),
\)
belongs to a given interval with high probability. Specifically, we consider the case where the closed loop polynomial must belong to the family,
\[
    z^3+[1,3]z^2+[1,3]z+[1,3],
\]
with probability $1-\varepsilon$.
It can be verified that this is equivalent to finding a feasible solution of the chance constraint,
{\small
\[
    \mathbb{P}_{\xi}(1\leq x^\top \big(A^{(i)}\xi + b^{(i)}\big) + \xi^\top c^{(i)} + d^{(i)}\leq 3, \; i=1,2,3)\geq 1-\varepsilon.
\]}
This can be reduced to the form shown in~\eqref{eq:ccp} by appropriately defining new decision variables fixed at $1$ and by exploiting the property that any affine transformation of $\xi$ also satisfies Assumption~\ref{assumption:rv}. 

In our experiment, we consider the case where $\xi$ follows a multivariate normal distribution with mean $\mu = (0, 0, 0, 0)$ and covariance matrix $\Sigma = \mathop{\text{diag}}(0.0278, 0.0069, 0.0069, 0.0069)$.
We compare the performance of our scaled scenario approach parameterized by scaling factors $s=1.1$ and $s=1.2$,
with the classical scenario approach~\cite{campi2009scenario}
across different constraint violation levels $\varepsilon \in \{10^{-3}, 10^{-4}, 10^{-5}\}$.
For all methods, we set the confidence level, $\beta=0.05$.
The experiments were conducted in Julia~1.8.1 using JuMP~1.23.3 and Gurobi~11.0.3
    as the solver with a time limit of 1~hour.
    All experiments were run on a 2.8~GHz Intel Xeon Processor with 12~GB RAM.

Figure~\ref{fig:RPA_CPU} compares the computational time for the three methods. For each $\varepsilon$ level, we conduct 100~independent trials to account for variability introduced by random sampling. The results are presented as box plots to visualize the distribution of computational times across these trials.
We find that as $\varepsilon$ decreases, indicating higher reliability requirements, our proposed scaled method achieves significantly improved computational efficiency compared to the classical approach.
Notably, for $\varepsilon=10^{-5}$, the solver fails to find a feasible solution of the latter model within the time limit due to numerical difficulties, whereas the former with $s=1.1$ and $s=1.2$ is solved in approximately 3,500 and 65 seconds, respectively.
The computational advantages become even more pronounced for smaller $\varepsilon$, highlighting the practical scalability of our approach for problems requiring very high levels of robustness.
\begin{figure}[!ht]
    \centering
    \includegraphics[width=0.4\textwidth]{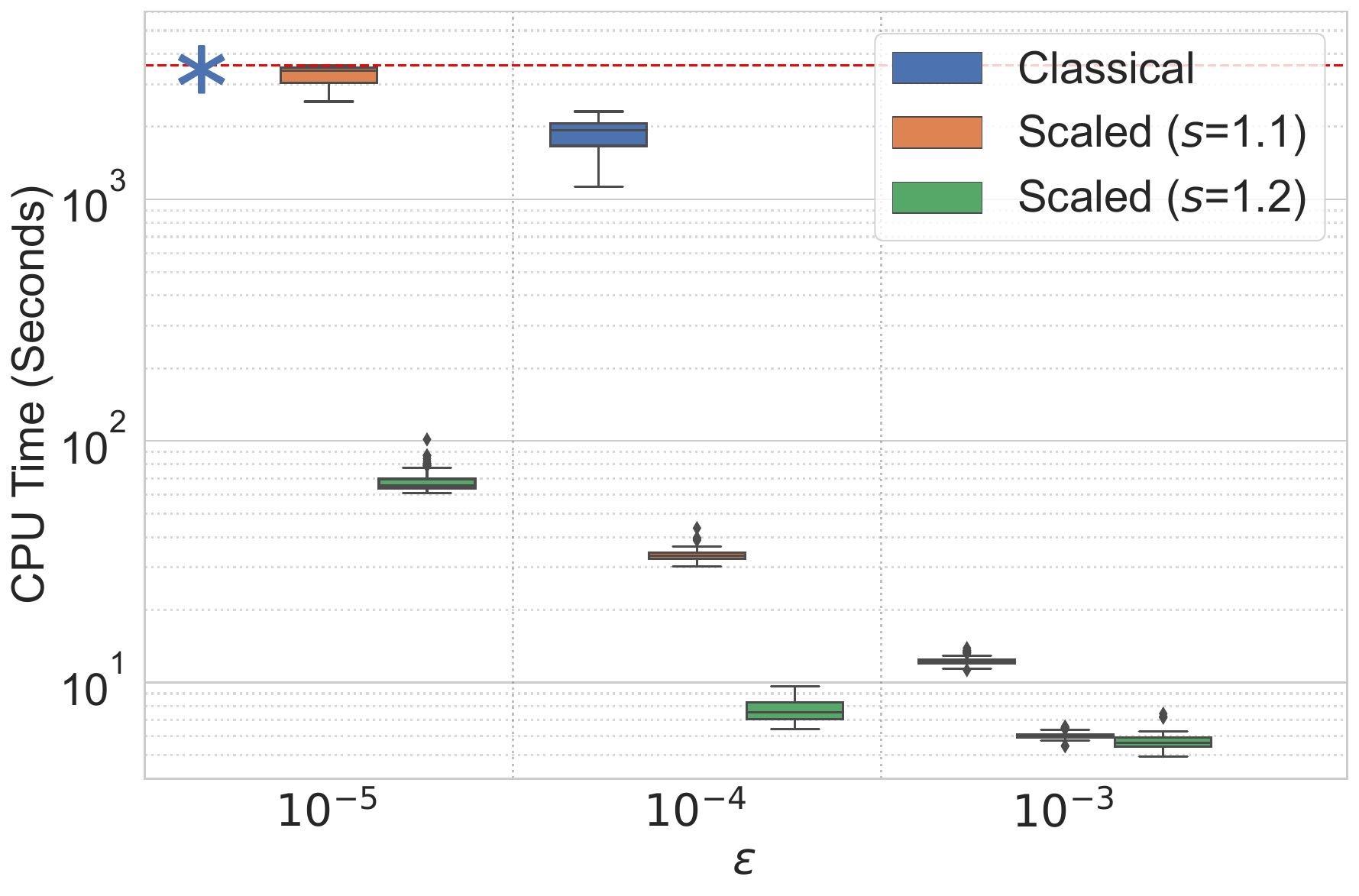}
    \caption{Comparison of computational times. The `*' indicates that the solver could not find any feasible solution in 1~hour using the classical method at $\varepsilon = 10^{-5}$.}
    \label{fig:RPA_CPU}
\end{figure}

Figures~\ref{fig:RPA_Obj} and~\ref{fig:RPA_VP} illustrates the objective values and violation probabilities, respectively, of the solutions obtained by the three methods. The violation probabilities are evaluated using
    $10^9$ out-of-sample Monte Carlo realizations.
As before, we conducted 100~independent trials in each case to account for variability due to sampling.
The results reveal that our proposed scaled scenario approach produces solutions with lower violation probabilities than the desired target $\varepsilon$ and with higher objective values. The conservatism tends to increase with the scaling factor $s$. However, this additional conservatism can be advantageous in the context of robust control design for safety-critical applications. Moreover, the scaling factor $s$ serves as a tunable hyperparameter, offering a balance between computational efficiency and solution robustness, allowing practitioners to adapt the method to their specific needs.

\begin{figure}[!ht]
    \centering
    \includegraphics[width=0.4\textwidth]{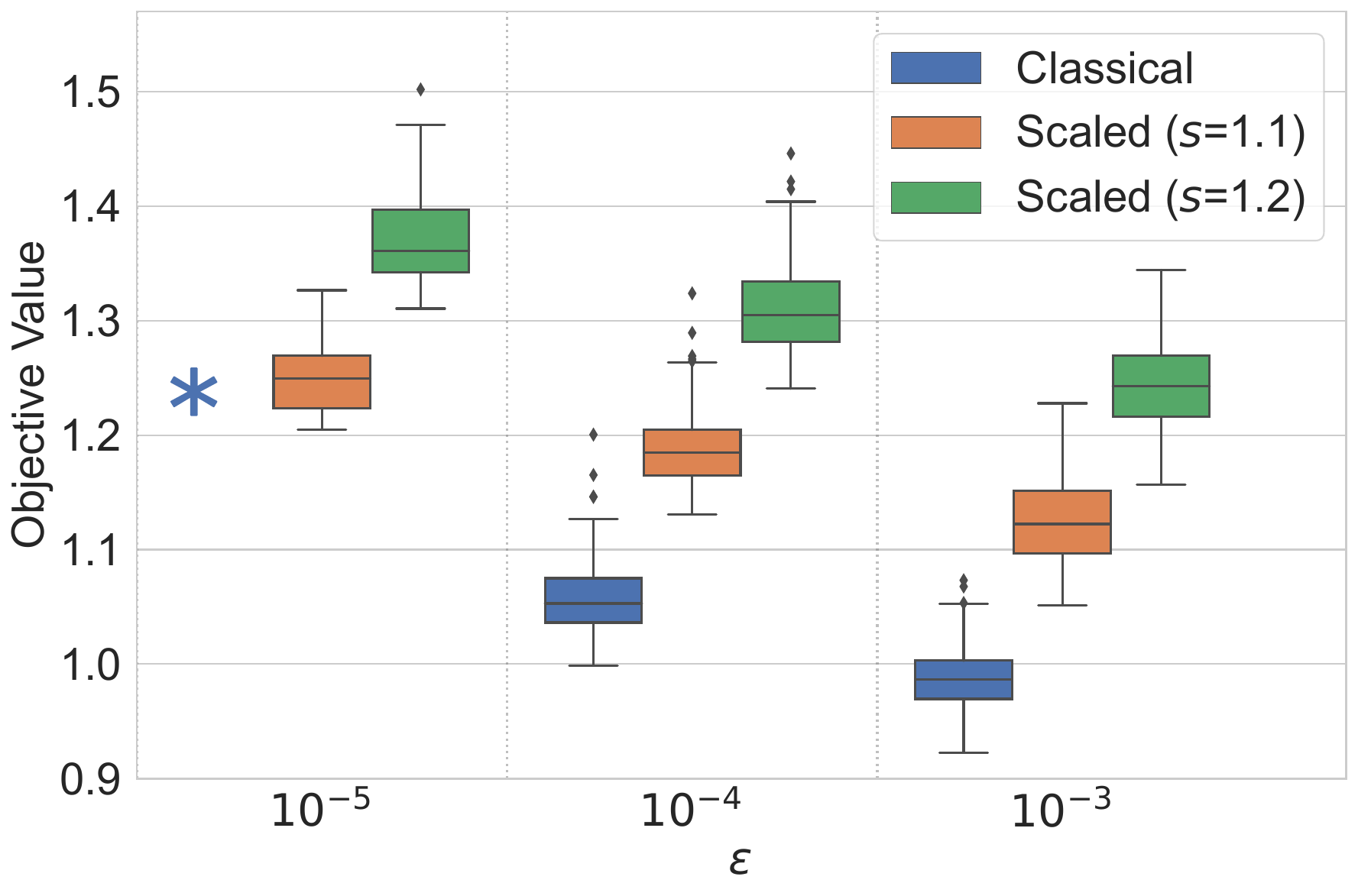}
    \caption{Comparison of objective values.}
    \label{fig:RPA_Obj}
\end{figure}

\begin{figure}[!ht]
    \centering
    \includegraphics[width=0.4\textwidth]{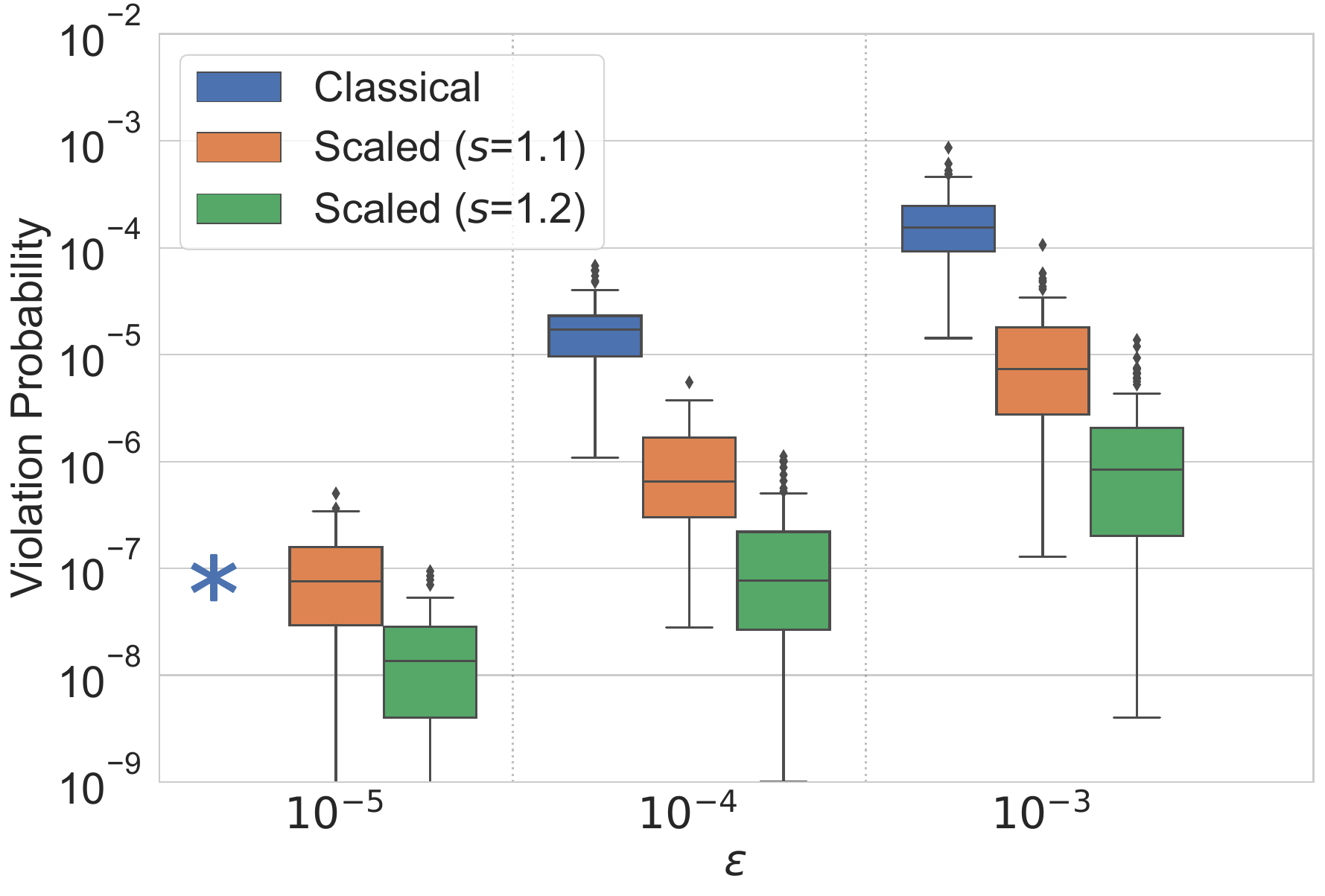}
    \caption{Comparison of violation probabilities.}
    \label{fig:RPA_VP}
\end{figure}

\bibliographystyle{IEEEtran}
\bibliography{references.bib}

\appendix

The proof of Theorem~\ref{thm:weak_feasibility} requires the following technical lemma that establishes a uniform large deviations principle for the family of distributions satisfying Assumption~\ref{assumption:rv}.
Recall that 
    $q(u) = h(u)u^{\alpha}$, where $h$ is a slowly varying function, and $\lambda$ is the limiting function satisfying equation~\eqref{eq:lambda}.
\begin{lemma}\label{tailprob}
    Suppose $\xi$ satisfies Assumption \ref{assumption:rv}, $E \subseteq \mathbb{R}^n$ is any compact subset of $\mathbb{R}^n$, and $y \in E$. Then, the following limit holds uniformly over $E$.
    \begin{equation}\label{eq:unif_conv_prob}
        \lim_{u\to\infty} \dfrac{\log\mathbb{P} \left(\max\limits_{i=1,\ldots,m}y^\top A^{(i)} \xi > u \right)}{-q(u)\min\limits_{i=1,\ldots,m} I \left(y^\top A^{(i)} \right)}
        = 1,
    \end{equation}
    where, for $i = 1, \ldots, m$, we define
    \begin{equation}
        I(y^\top A^{(i)}) \coloneqq \inf_{z \in \mathbb{R}^d_{+}} \left\{\lambda(z): y^T A^{(i)}z\geq 1 \right\}.\label{eq:unif_conv}
    \end{equation}
\end{lemma}
\begin{proof}
    Under Assumption~\ref{assumption:rv}, the proof of \cite[Theorem~1]{deo2023achieving} shows that the random vectors $\{\xi/u : u>0\}$ satisfy, for any closed $F$ and open $G$ subsets of $\mathbb{R}^d$, the bounds
    \begin{align}
        \limsup_{u\to\infty}\frac{\log \mathbb P(\xi/u \in F)}{q(u)} & \leq  -\inf_{z\in F} \lambda(z) \text{ and } \label{eq:ldp-ub} \\
        \liminf_{u\to\infty}\frac{\log \mathbb P(\xi/u \in G)}{q(u)} & \geq  -\inf_{z\in G}\lambda(z) \label{eq:ldp-lb}
    \end{align}
    Now, let $\{y_u\}$ denote any sequence converging to $y$, which we shall denote as $y_u \to y$. Then, observe that $\max_{i=1,\ldots,m} y_u^\top A^{(i)} z_u \to \max_{i=1,\ldots,m} y^TA^{(i)} z $, whenever $z_u\to z$. Thus, the sequence of functions $w_u(z) = \max_{i=1,\ldots,m} y_u^\top A^{(i)} z$ converges continuously to $w(z) = \max_{i=1,\ldots,m} y^\top A^{(i)} z$. Hence, \cite[Corollary 1]{deo2023achieving} implies that for any $(M,\delta)>0$, there exists $u_0$ such that for all $u>u_0$,
    \begin{equation}\label{eq:containments}
        \lev_{1,M}(w_u) \subseteq \olsi{\lev_1(w)} + B_{\delta}, \lev_{1+\delta,M}(w) \subseteq \lev_{1,M}(w_u),
    \end{equation}
    where $B_\theta \coloneqq \left\{z \in \mathbb{R}^d: \|z\| \leq \theta \right\}$ for any $\theta > 0$.
    Observe that the probability in \eqref{eq:unif_conv_prob} can also be equivalently written as
    \[
        \mathbb{P} \left(\max\limits_{i=1,\ldots,m}y^\top A^{(i)} \xi > u \right)
        =
        \mathbb P(\xi/u \in \lev_1(w))
    \]
    The rest of the proof relies on establishing the following upper and lower bounds for the latter probability.

    \noindent \emph{Upper Bound.} First, note that
    \[
        \mathbb P(\xi/u \in \lev_1(w_u)) \leq  P(\xi/u \in\lev_{1,M}(w_u)) + P(\xi/u \not \in B_M).
    \]
    Then, an application of inequality~\eqref{eq:ldp-ub}, \cite[Lemma 1.2.15]{dembo2009large}, and the first containment in \eqref{eq:containments} give:
    \begin{align*}
        \limsup_{u\to\infty} \frac{\log \mathbb P(\xi/u\in  \lev_1(w_u))}{q(u)} & \\ {
        \leq-\min\left\{\inf_{z\in \lev(w)+B_\delta} \lambda(z), \inf_{\|z\|\geq M}\lambda(z)\right\}.}
    \end{align*}
    Since $\lambda$ is a positively homogeneous function, $\lambda(z) \to \infty$ as $\|z\|\to\infty$ and hence, the second term inside $\min\{\cdot, \cdot\}$ goes to $\infty$. Then,
    since $\delta$ and $M$ above were arbitrary, we have
    \[
        \limsup_{u\to\infty} \frac{\log \mathbb P(\xi/u\in  \lev_1(w_u))}{q(u)}  \leq - \inf_{z\in \lev_1(w)} \lambda(z)
    \]

    \noindent \emph{Lower Bound.} The lower bound~\eqref{eq:ldp-lb} and the second containment in~\eqref{eq:containments} give:
    \[
        \liminf_{u\to\infty} \frac{\log \mathbb P(\xi/u\in  \lev_1(w_u))}{q(u)} \geq  - \inf_{z\in \lev_{1+\delta,M}(w)}\lambda(z).
    \]
    As before, since $\delta$ and $M$ above are arbitrary, we have:
    \[
        \liminf_{u\to\infty} \frac{\log \mathbb P(\xi/u\in  \lev_1(w_u))}{q(u)} \geq  - \inf_{z\in \lev_{1}(w)}\lambda(z).
    \]
    Finally, let $y_u\to y$. Then, observe that
    \begin{equation}\label{eq:continous_convergence}
        \lim_{u\to\infty} \frac{\log \mathbb P\left(\max\limits_{i=1,\ldots,m}  y_u^TA^{(i)} \xi > u \right)}{q(u)} = -\min_{i=1,\ldots,m} I(y^TA^{(i)}).
    \end{equation}
    To conclude the proof denote the left hand side as a sequence of functions $\phi_u$ and the right hand side as a limiting function $\phi$. Then \eqref{eq:continous_convergence} implies that $\phi_u\to \phi$ continuously and therefore uniformly over compact subsets~\cite[Theorem 7.14]{rockafellar2009variational}. 
\end{proof}

\noindent\emph{Proof of Theorem~\ref{thm:weak_feasibility}.}
Define $y_{\varepsilon} = x_{\varepsilon} q^{-1} \left( \log \frac{1}{\varepsilon} \right)$, where for sufficiently small $\varepsilon$, we note that $\log \frac{1}{\varepsilon}$ is sufficiently large and the inverse $q^{-1}$ is well-defined.
Now, note the following:
\begin{align}
      & \liminf_{\varepsilon \rightarrow 0^{+}}\frac{\log V(x_{\varepsilon})}{\log \varepsilon} \nonumber                                                                            \\
    = & \liminf_{\varepsilon \to 0^{+}}\frac{\log \mathbb{P} \left( \max_{i=1,\ldots,m} x_{\varepsilon}^\top A^{(i)} \xi > 1 \right)}{\log \varepsilon } \nonumber                                                                      \\
    = & \liminf_{\varepsilon \to 0^{+}} \frac{\log \mathbb{P} \left( \max_{i=1,\ldots,m} y_{\varepsilon}^\top A^{(i)} \xi > q^{-1} \left( \log \frac{1}{\varepsilon} \right) \right)}{\log \varepsilon } \nonumber                                                           \\
    = & \liminf_{\varepsilon \to 0^{+}}\frac{\log \mathbb{P} \left( \max_{i=1,\ldots,m} y_{\varepsilon}^\top A^{(i)} \xi > q^{-1} \left( \log \frac{1}{\varepsilon} \right) \right)}{-q \left( q^{-1} \left( \log \frac{1}{\varepsilon} \right) \right)\min_{i=1,\ldots,m} I(y_{\varepsilon}^\top A^{(i)})} \nonumber \\
      & \qquad\;\;\; \cdot \frac{-q \left( q^{-1} \left( \log \frac{1}{\varepsilon} \right) \right)\min_{i=1,\ldots,m} I(y_{\varepsilon}^\top A^{(i)})}{\log \varepsilon } \nonumber                                                                                                         \\
    = & \liminf_{\varepsilon \to 0^{+}}\min_{i=1,\ldots,m} I(y_{\varepsilon}^T A^{(i)}) \label{eq:rate_temp}                 
\end{align}
where we have used Lemma~\ref{tailprob} for the first fraction, and the fact that $-q \left( q^{-1} \left( \log \frac{1}{\varepsilon} \right) \right) = \log \varepsilon$ for the second fraction.

Now, since $x_{\varepsilon}$ is a feasible solution of \eqref{eq:s-sp} with $N \geq \ubar{N}(\varepsilon^{s^{-\alpha}},\beta)$, \cite[Theorem 1]{campi2009scenario}
implies that $x_{\varepsilon}$ is an $\varepsilon^{s^{-\alpha}}$-level robustly feasible solution of the scaled problem,
\begin{equation}
    \label{prob1}
    \mathbb{P} \left( \max_{i=1,\ldots,m} y_{\varepsilon}^\top A^{(i)} \xi > \frac{q^{-1} \left( \log \frac{1}{\varepsilon} \right)}{s} \right) \leq \varepsilon^{s^{-\alpha}},
\end{equation}
with probability at least $1 - \beta$. 
Taking logarithms on both sides of~\eqref{prob1} and rearranging:
\begin{equation}
    \label{prob2}
    \frac{\log \mathbb{P} \left( \max_{i=1,\ldots,m} y_{\varepsilon}^\top A^{(i)} \xi > \frac{q^{-1} \left( \log \frac{1}{\varepsilon} \right)}{s} \right)}{\log \varepsilon^{s^{-\alpha}}} \geq 1
\end{equation}
This is equivalent to:
\begin{align*}
    \label{limprob}
    \lim_{\varepsilon \to 0^{+}} & \frac{\log \mathbb{P} \left( \max_{i=1,\ldots,m} y_{\varepsilon}^\top A^{(i)} \xi > \frac{q^{-1} \left( \log \frac{1}{\varepsilon} \right)}{s} \right)}
    {-q \left( \frac{q^{-1} \left( \log \frac{1}{\varepsilon} \right)}{s} \right) \min_{i=1,\ldots,m} I(y_{\varepsilon}^\top A^{(i)})} \nonumber                                                                                                        \\
                                              & \; \cdot \frac{-q \left( \frac{q^{-1} \left( \log \frac{1}{\varepsilon} \right)}{s} \right) \min_{i=1,\ldots,m} I(y_{\varepsilon}^\top A^{(i)})}{\log \left( \varepsilon^{s^{-\alpha}} \right)} \geq 1.
\end{align*}

To proceed, note that the collection $\{x_\varepsilon\}_{\varepsilon>0}$ that are feasible to \eqref{eq:s-sp} lie in the compact set $X$. Therefore, using Lemma \ref{tailprob}, the first fraction above converges to 1 as $\varepsilon \to 0$. For the second fraction, we use the definition of $q(u) = h(u) u^\alpha$ and that $h(\cdot)$ is slowly varying:
\begin{align*}
    \lim_{\varepsilon \to 0^{+}} \frac{-q \left( \frac{q^{-1} \left( \log \frac{1}{\varepsilon} \right)}{s} \right)}{\log \varepsilon^{s^{-\alpha}}}
    = \lim_{\varepsilon \to 0^{+}} \frac{-q \left( q^{-1} \left( \log \frac{1}{\varepsilon} \right) \right) \cdot \frac{1}{s^\alpha}}{\log \varepsilon^{s^{-\alpha}}} = 1
\end{align*}
Together, this implies:
\[
    \liminf_{\varepsilon \to 0^{+}} \min_{i=1,\ldots,m} I(y_{\varepsilon}^\top A^{(i)}) \geq 1.
\]
Along with equation~\eqref{eq:rate_temp}, this proves the claimed result.
\end{document}